\documentclass[12pt]{article}
\usepackage{amsthm}
\usepackage{amssymb}
\usepackage{amsmath}
\usepackage{amsfonts}
\usepackage{times}

\usepackage[paper=a4paper, top=2.5cm, bottom=2.5cm, left=2.5cm, right=2.5cm]{geometry}

\def\qed{\hfill $\vcenter{\hrule height .3mm
\hbox {\vrule width .3mm height 2.1mm \kern 2mm \vrule width .3mm
height 2.1mm} \hrule height .3mm}$ \bigskip}

\def \RR {\mathbb R}

\def \EE {\mathbb E}

\def \PP {\mathbb P}
\def \eps {\varepsilon}

\def \Id {\mathrm{Id}}

\def\DD{\mathcal{D}}

\newcommand{\remove}[1]{}

\newtheorem{theorem}{Theorem}
\newtheorem{lemma}[theorem]{Lemma}

\newtheorem{fact}[theorem]{Fact}
\newtheorem{claim}[theorem]{Claim}

\newtheorem{proposition}[theorem]{Proposition}
\newtheorem{corollary}[theorem]{Corollary}
\theoremstyle{definition}

\theoremstyle{remark}

\long\def\symbolfootnotetext[#1]#2{\begingroup
\def\thefootnote{\fnsymbol{footnote}}\footnotetext[#1]{#2}\endgroup}

\begin{document}
\title{How many matrices can be spectrally balanced simultaneously?}
\author{Ronen Eldan \thanks{Weizmann Institute of Science; \texttt{ronen.eldan@weizmann.ac.il}}
\and
 Fedor Nazarov \thanks{Kent State University; \texttt{fedja@math.msu.edu}}
 \and
 Yuval Peres \thanks{Microsoft Research; \texttt{peres@microsoft.com}}}
\maketitle

\begin{abstract}
We prove that any $\ell$ positive definite $d \times d$ matrices, $M_1,\ldots,M_\ell$, of full rank, can be simultaneously spectrally balanced in the following sense: for any $k < d$ such that $\ell \leq \lfloor \frac{d-1}{k-1} \rfloor$, there exists a matrix $A$ satisfying $\frac{\lambda_1(A^T M_i A) }{ \mathrm{Tr}( A^T M_i A ) } < \frac{1}{k}$ for all $i$, where $\lambda_1(M)$ denotes the largest eigenvalue of a matrix $M$. This answers a question posed by Peres, Popov and Sousi (\cite{PPS}) and completes the picture described in that paper regarding sufficient conditions for transience of self-interacting random walks.
 Furthermore, in some cases we give quantitative bounds on the transience of such walks.
\end{abstract}

\section{Introduction}

The main objective of this note is to address the following question, raised in \cite{PPS}: given a set of $\ell$ symmetric bilinear positive-definite forms acting on a $d$-dimensional linear space, under which basis of this space are the corresponding matrices as spectrally balanced as possible, in the sense that the ratio between the trace and the operator norm is maximal?

 For a square matrix $M$, denote by $\lambda_i(M)$
the $i$-th eigenvalue of $M$ in decreasing order (counting with multiplicity) and by $\mathrm{Tr}(M)$   the sum of the eigenvalues of $M$. Recall that for any positive definite symmetric matrix $M$, the operator norm
$\Vert M \Vert_{OP}$ (with respect to its action on Euclidean space) coincides with $\lambda_1(M)$.  Our main theorem reads:

\begin{theorem} \label{mainthm}
Let $k,d,\ell$ be positive integers such that $d > k$ and $\ell \leq \lfloor \frac{d-1}{k-1} \rfloor$. Let $M_1,\ldots,M_{\ell}$ be $d \times d$ positive-definite symmetric matrices of full rank. Then there exists a matrix $A$ such that
\begin{equation} \label{eqmain}
\frac{\lambda_1(A^T M_i A) }{ \mathrm{Tr}( A^T M_i A ) } < \frac{1}{k}
\end{equation}
for all $1 \leq i \leq \ell$.
\end{theorem}

The bound given by this theorem is sharp in the following sense: if $k, \ell, d$ satisfy $(k-1) \ell \geq d$ then there exist $d \times d$  symmetric positive definite invertible matrices $\{M_i\}_{i=1}^\ell$, such that  for every $d \times d$ matrix $A$, there is an  $ i \in [1, \ell]$ that satisfies
\begin{equation} \label{eqmainsharp}
\frac{\lambda_1(A^T M_i A) }{ \mathrm{Tr}( A^T M_i A ) } > \frac{1}{k}.
\end{equation}
This is shown in subsection \ref{sec:sharp} below. \\

To illustrate Theorem \ref{mainthm},  consider the case of two $3 \times 3$ positive definite invertible matrices $M_1,M_2$. We claim that there exists a matrix $A$  such that  $\mathrm{Tr}(A^T M_i A) > 2 \Vert M_A \Vert_{OP}$ for $i=1,2$. To see this, first remark that since $M_1$ is symmetric, we can apply a rotation
 to the $M_i$  (i.e., multiply each $M_i$ by the rotation from the right and by its transpose from the left) to transform $M_1$  to a diagonal matrix. Applying a suitable diagonal matrix and another rotation, we can transform $M_1$ to the identity $I$, and $M_2$  to a diagonal matrix $M_2'=\mathrm{diag}(a,b,c)$.  We may assume that $ a \ge b \ge c>0$. Applying the diagonal matrix $\mathrm{diag}(\sqrt{b/a}, 1, 1 )$ to $I$ and $M_2'$ yields two matrices $M_1'', M_2''$, such that each $M_i''$ has two equal eigenvalues, and one positive eigenvalue which is smaller than or equal to these two. This proves our claim. \\

In \cite{PPS}, it is demonstrated how the question addressed in  Theorem \ref{mainthm} arises from the topic of \emph{self-interacting random walks}: Given  centered measures $\mu_1,\ldots, \mu_\ell$, one can consider a generalized random walk, where the law of each step (given the history) is  one of these measures; in each step, we choose an index $1 \leq i \leq \ell$ using a certain \emph{adapted rule}, and then take a step distributed according to $\mu_i$.
More precisely, let $\mu_1,\ldots,\mu_\ell$ be centered probability measures in $\RR^d$, with finite third moments. Let $\{\mathcal{F}_t\}$ be a filtration and for each $j=1,\ldots,\ell$ let $\{\xi_t^j\}_{t=1}^\infty$ be an independent sequence of random vectors with law $\mu_j$, adapted to this filtration
(i.e., $\xi_t^j$ is $\mathcal{F}_t$--measurable for every $t \ge 1$). We say that $\{X_t\}$ is an  \emph{adaptive random walk using the measures $\mu_1,\ldots,\mu_k$} if there exists an  $\{\mathcal{F}_t\}$-adapted  process $\{I_t\}$,   such that for all $t \ge 0$
$$
X_{t+1} = X_t + \xi_{t+1}^{I_t} \,.
$$

The following question was originally raised by I. Benjamini: Is there a simple sufficient condition on the centered measures $\{\mu_i\}_{i=1}^\ell$, that implies every adaptive walk using these measures is transient? In \cite{PPS} such a condition is given  in terms of the spectrum of the corresponding covariance matrices. Let $M_i$ denote the covariance matrix of $\mu_i$ for each $i$. Then the following statement holds.

\begin{theorem} \cite[Theorem 1.3]{PPS}
Suppose that there exists a matrix $A$ such that for all $i=1,\ldots,\ell$, one has $\mathrm{Tr}(A^T M_i A) > 2 \lambda_1(A^T M_i A)$. Then every adaptive random walk using $\mu_1,\ldots,\mu_\ell$, is transient.
\end{theorem}

The condition given by this theorem is not entirely satisfactory, as it is not clear when such a matrix $A$ exists. This loose end is tied by Theorem \ref{mainthm} above, yielding the immediate corollary
\begin{corollary}
Suppose that $d > \ell$ and the covariance matrices $M_1,\ldots,M_\ell$ of the centered measures $\mu_1,\ldots,\mu_\ell$ are invertible. Then any adaptive random walk using $\mu_1,\ldots,\mu_\ell$ is transient.
\end{corollary}

Theorem \ref{mainthm} actually yields a quantitative estimate on the probability of returning to a neighborhood of the origin:  
\begin{theorem} \label{cor1}
Let $k,d,\ell$ be positive integers such that $d > k>1$ and $\ell \leq \lfloor \frac{d-1}{k-1} \rfloor$. Let $\mu_1,\ldots,\mu_\ell$ be centered probability measures in $\RR^d$, with finite second moments and non-singular covariance matrices. Then there exists  $\eps > 0$, so that for every $R>0$ there is a constant $C_R>0$ with the following property:  Every adaptive random walk $\{X_t\}$   using $\mu_1,\ldots,\mu_\ell$ satisfies
$$
\PP \bigl (\exists t>T \mbox{ such that } X_t \in B(0,R) \bigr ) \leq C_R T^{-\frac{k-2}{2}-\eps}, ~~ \forall T>0.
$$
\end{theorem}
For the proof, see Section 3.

\section{Proof of Theorem \ref{mainthm}}

Let $(\mathbb{M}_{d \times d}, \mathcal{T}) $ be the space of $d \times d$ real matrices equipped with the standard topology $\mathcal{T}$ of $\RR^{d^2}$. For a matrix $A \in \mathbb{M}_{d \times d}$, define $s_i(A) = \sqrt{\lambda_i(A^T A)}$, the $i$-th singular value of $A$. Next, define
$$
\DD_R = \left \{ A \in \mathbb{M}_{d \times d}; ~s_1(A) = 1 \mbox { and } ~ 1 \leq \frac{s_j(A)}{s_{j+1}(A)} \leq R, ~ \forall  j \in [1, d-1] \right  \}.
$$
We begin by recalling the following elementary fact (see e.g., \cite{Kato})
\begin{fact} \label{fact:cont}
The functions $s_i(\cdot)$ are continuous with respect to the topology $\mathcal{T}$.
\end{fact}

As a corollary, we have:
\begin{corollary}
The set $\DD_R$ is compact with respect to the topology $\mathcal{T}$.
\end{corollary}
\begin{proof}
Consider the set $\DD_R' = \left \{ A \in \mathbb{M}_{d \times d}; 1 \geq s_1(A) \geq \ldots \geq s_d(A) \geq R^{-d} \right \}$. It is clear that $\DD_R'$ is bounded. Moreover, the two conditions $s_1(A) \leq 1$ and $s_d(A) \geq R^{-d}$ are closed conditions and therefore $\DD_R'$ is compact. Now, by definition for every $1 \leq j \leq d-1$,  we have that $s_{j+1}(A)$ is strictly bounded away from zero and therefore the expression $\frac{s_j(A)}{s_{j+1}(A)}$ is continuous in $\DD_R'$. Consequently, the conditions which appear in the definition of $\DD_R$ are closed conditions and the corollary follows.
\end{proof}

Consider the functions $f_i: \DD_R \to \RR$,
$$
f_i(A) = \frac{\lambda_1(A^T M_i A)}{\mathrm{Tr}(A^T M_i A)}.
$$
and
$$
f(A) = \max_{1 \leq i \leq \ell} f_i(A).
$$

\begin{claim}
The functions $\{f_i\}_{1 \leq i \leq \ell}$ and $f$ are continuous in the domain $\DD_R$.
\end{claim}
\begin{proof}
  Fact \ref{fact:cont} and the continuity of matrix multiplication imply that the functions $A \to \lambda_1(A^T M_i A)$ and $A \to \mathrm{Tr}(A^T M_i A)$ are continuous.

It remains to prove that $\mathrm{Tr}(A^T M_i A)$ is bounded away from zero on $\DD_R$. Fix $1 \leq i \leq \ell$. Recall that $M_i$ is of full rank. Therefore there exists a constant $c>0$ such that $\langle v, M_i v \rangle \geq c |v|^2$ for all $v \in \RR^d$. By the definition of $\DD_R$ we have $s_d(A) \geq R^{-d}$ for all $A \in \DD_R$. Therefore,  
$$
\langle A^T M_i A v, v \rangle = \langle M_i A v, A v \rangle \geq c |A v|^2 \geq c R^{-2d} |v|^2
$$
and it follows that $\mathrm{Tr}(A^T M_i A) \geq c R^{-2d}$. 
\end{proof}

By the compactness of $\DD_R$ and the continuity of $f$, we deduce that the minimum $\min_{A \in \DD} f(A)$ is attained by some matrix $A \in \DD_R$. \\

Next, we claim that without loss of generality, we may assume that the matrix $A$ is symmetric positive definite. Indeed, using the polar decomposition theorem (\cite[Chapter 3]{HJ}) we can write $A = B U$ where $B$ is positive semi-definite and $U$ is an orthogonal matrix. Clearly $s_j(B)=s_j(A)$ for all $j$. Moreover,
$$
\lambda_j \bigl (A^T M_i A \bigr ) = \lambda_j \bigl (U^T B M_i B U \bigr ) = \lambda_j(B M_i B), ~~ \forall i \in [1,\ell], ~ \forall j \in [1,d]\, .
$$
It follows that $B \in \DD_R$ and that $f(B) = f(A)$. Therefore, by taking $B$ in place of $A$, we see that our assumption is justified. \\

The proof will rely on a perturbative argument: we will assume that the $d \times d$ matrix $A \in  \DD_R$ satisfies
\begin{equation} \label{Aisbad}
f(A) \geq \frac{1}{k} \,,
\end{equation}
and deduce that there exists a small perturbation of $A$ which lies in $\DD_R$ and decreases $f$. This will imply that the $A \in  \DD_R$ that minimizes $f$
must satisfy $f(A)<1/k$, as desired.

Our first goal is to reduce the proof to the case where $A \in \partial \DD_R$ (where $\partial$ denotes the boundary with respect to the topology $\mathcal{T}$). To this end, assume that $A$ is in the interior of $\DD_R$. Define
$$
I = \{i \in [1,\ell] ; ~ f_i(A) = f(A) \}
$$
and denote $\ell' = |I|$ so that $\ell' \leq \ell$. For each $i \in I$ and $1 \leq j \leq k-1$, let $q_{i,j}$ be a unit eigenvector corresponding to $\lambda_j(A M_i A)$ chosen so that $q_{i,j_1} \perp q_{i, j_2}$ for $j_1 \neq j_2$. Let $v$ be a unit vector satisfying
$$
v \perp q_{i,j}, ~~ \forall i \in I, ~  \forall j \in [1, k - 1].
$$
Such a vector exists thanks to the assumption $\ell \leq \frac{d-1}{k-1}$. Define
\begin{equation} \label{basicconstr}
A(\eps) = A (\mathrm{Id} + \eps v \otimes v).
\end{equation}
By definition, we have for all $i \in I$ and for all $  j \in [1,k-1]$ that
$$
A(\eps)^T M_i A(\eps) q_{i,j}
$$
$$
= A M_i A q_{i,j} + \eps A(\eps)^T M_i A  (v \otimes v ) q_{i,j} + \eps (v \otimes v) A M_i A q_{i,j}
$$
$$
= A M_i A q_{i,j} = \lambda_j(A M_i A) q_{i,j}.
$$
which means that $\lambda_1(A M_i A), \ldots, \lambda_{k-1}(A M_i A)$ are eigenvalues of the matrix $A(\eps)^T M_i A(\eps)$ for all $\eps \in \RR$ (however, those are not necessarily the largest $k-1$ eigenvalues of this matrix). Observe that since $k < d$, the assumption \eqref{Aisbad} implies that $\lambda_k(A M_i A) < \lambda_1(A M_i A)$ for all $i \in I$. By continuity, this means that there exists some $\eps_0 > 0$ such that the above are in fact the $k-1$ \emph{largest} eigenvalues of the matrix $A(\eps)^T M_i A(\eps)$ as long as $0 < \eps < \eps_0$.

In other words, there exists some $\eps_0 > 0$ such that for all $i \in I$, the function $\lambda_1(A(\eps)^T M_i A(\eps))$ is constant in the interval $[0, \eps_0]$. On the other hand, we have for all $u \perp v$ that
$$
\langle A(\eps)^T M_i A(\eps)u, u \rangle = \langle A M_i A (\mathrm{Id} + \eps v \otimes v) u, (\mathrm{Id} + \eps v \otimes v) u \rangle = \langle A M_i A u, u \rangle
$$
and
$$
\langle A(\eps)^T M_i A(\eps)v, v \rangle = (1 + \eps)^2 \langle A M_i A v, v \rangle
$$
which gives
$$
\mathrm{Tr}(A(\eps)^T M_i A(\eps)) = \mathrm{Tr}(A M_i A) + (2 \eps + \eps^2) \langle A M_i A v, v \rangle.
$$
Since the matrices $A M_i A$ are non-degenerate, it follows that the expression $\mathrm{Tr}(A(\eps)^T M_i A(\eps) )$ is strictly increasing with respect to $\eps$ on the interval $[0, \eps_0]$. We conclude that by choosing $\eps$ small enough (which ensures also that $A(\eps) \in \DD_R$), one can get $f(A(\eps)) < f(A)$ thus reaching a contradiction. \\

At this point we have reduced the proof to the case that $A \in \partial \DD_R$. The argument for this case is more delicate, as the direction of the perturbation must be chosen carefully to ensure that the perturbed matrix remains in $\DD_R$. Before we explain the idea which will allow us to do so, we will need a few more definitions.

For $\eta \in \RR^d$ and $\eps > 0$, we define
$$
A_\eta(\eps) = A (\mathrm{Id} + \eps \eta \otimes \eta).
$$
Let $D \subset \RR^d$ be defined as
$$
D = \left \{\eta \in \RR^d; ~ \exists \eps_0 > 0 \mbox{ such that } \frac{A_\eta(\eps)}{s_1(A_\eta(\eps))} \in \DD_R, ~ \forall \eps \in (0,\eps_0) \right \}.
$$
Finally, let $E_j$ be the subspace corresponding to the eigenvalue $\lambda_j(A)$.

The central tactic that will help us ensure that the perturbation of $A$ stays inside $\DD_R$ will be to consider vectors $\eta$ which satisfy the following condition
\begin{equation} \label{condeta}
| P_{E_j} \eta |^2 < \frac{1}{2} |P_{E_{j+1}} \eta |^2, ~~ \forall j \in \mathcal{J},
\end{equation}
where
$$
\mathcal{J} = \{j ; ~ \lambda_j(A) = R \lambda_{j+1} (A)  \}.
$$
The significance of this condition will be clarified by the following lemma:
\begin{lemma} \label{lemcondeta}
Whenever \eqref{condeta} holds, one has $\eta \in D$.
\end{lemma}

Before we prove this lemma, we will need the following well-known facts regarding the dependence of eigenvalues of $A_\eta(\eps)$ on $\eps$.

\begin{lemma} \label{eigenperturb}
Let $B$ be a positive-definite $d \times d$ matrix. For all $j \in [1,d]$ write $\lambda_j = \lambda_j(B)$ and  let $E_j$ be the eigenspace corresponding to $\lambda_j$. Also, denote $E_0 = \{0\}$. Let $\eta \in \RR^d$ and define
$$
B(\eps) = (\mathrm{Id} + \eps \eta \otimes \eta)^T B (\mathrm{Id} + \eps \eta \otimes \eta) \,.
$$
Then, for every $j$ such that $E_{j-1} \neq E_j$, we have
\begin{equation} \label{eqeigenderiv}
\lambda_j ( B(\eps) ) = \lambda_j + 2 \eps \lambda_j |P_{E_j} \eta |^2 + O(\eps^2) \,,
\end{equation}
and there exists $\eps_0 > 0$ such that for every $j$ for which $E_{j-1} = E_j$, we have
$$
\lambda_j ( B(\eps) ) = \lambda_j, ~~ \forall \eps \in (0, \eps_0) \, .
$$
\end{lemma}
The proof of this lemma relies on a standard eigenvalue sensitivity analysis technique. For completeness, we provide this proof in the end of this section. As an immediate corollary, we get:

\begin{corollary} \label{eigenperturbcor}
Let $j \in [2,d]$. If $E_{j-1} \neq E_{j}$, then
$$
\lambda_j (A_ \eta(\eps)^T A_\eta(\eps)) = \lambda_j(A^2) + 2 \eps \lambda_j(A^2) |P_{E_j} \eta |^2 + O(\eps^2).
$$
Otherwise, we have
$$
\lambda_j (A_ \eta(\eps)^T A_\eta(\eps)) = \lambda_j(A^2) \, .
$$
\end{corollary}

We can now prove Lemma \ref{lemcondeta}.

\begin{proof}[\textbf{Proof of Lemma \ref{lemcondeta}}]
Define $B(\eps) = A_\eta(\eps)^T A_\eta(\eps)$. Remark that the first eigenvalue of $\frac{A_\eta(\eps)}{s_1(A_\eta(\eps))}$ is equal to $1$ whenever the denominator is non-zero, and by continuity, the latter is true whenever $|\eps|$ is smaller than some positive constant. We therefore deduce form the definition of the domain $\DD_R$ that in order to prove the lemma, it is enough to show that for a fixed $1 \leq j \leq d-1$, there exists $\eps_0 > 0$ such that
\begin{equation} \label{ntplem4}
\frac{s_{j} (A_\eta(\eps))^2 }{s_{j+1} (A_\eta(\eps))^2 } = \frac{\lambda_{j} (B(\eps)) }{\lambda_{j+1} (B(\eps))} \leq R^2, ~~ \forall  \eps \in (0,\eps_0) \, .
\end{equation}
If $\lambda_{j}(B(0)) < R^2 \lambda_{j+1}(B(0))$, then the above holds by the continuity of eigenvalues with respect to perturbations of the entries. Otherwise, we have $E_j \neq E_{j+1}$ and we may use Corollary \ref{eigenperturbcor} to get
$$
\log (\lambda_{j+1} (B(\eps))) \geq \log(\lambda_{j+1}(B(0))) + 2 \eps |P_{E_{j+1}} \eta |^2 - C_1 \eps^2
$$
for a constant $C_1 > 0$, provided $\eps>0$ is small enough. Another application of the same lemma gives
$$
\log(\lambda_{j} (B(\eps))) \leq \log(\lambda_{j}(B(0))) + 2 \eps |P_{E_{j}} \eta |^2 + C_2 \eps^2 <
$$
$$
\log(\lambda_{j}(B(0))) + \eps |P_{E_{j+1}} \eta |^2 + C_2 \eps^2
$$
for some constant $C_2>0$ and all sufficiently small $\eps>0$, where in the last inequality we used the condition \eqref{condeta}. A combination of these two inequalities gives
$$
\log \left ( \frac{\lambda_{j} (B(\eps))}{ \lambda_{j+1} (B(\eps)) } \right ) < \log \left ( \frac{\lambda_{j} (B(0))}{ \lambda_{j+1} (B(0)) } \right ) - \eps |P_{E_{j+1}} \eta |^2 + (C_1+C_2)\eps^2
$$
$$
= R^2 - \eps |P_{E_{j+1}} \eta |^2 + (C_1+C_2) \eps^2.
$$
By choosing $\eps_0$ small enough,
% = \frac{|P_{E_{j+1}} \eta |^2}{C_1+C_2}$, 
the inequality \eqref{ntplem4} is established and the lemma is proved.
\end{proof}
Next, we denote  $s_j = s_j(A)$ for all $  j \in [1, d]$, and fix an orthonormal basis $w_1,\ldots,w_d$ satisfying $A w_j = s_j w_j$
for all $  j \in [1, d]$. For a vector $\eta \in \RR^d$,  define
$$
u_j = u_j(\eta) = s_j \langle \eta, w_j \rangle, ~~ \forall j \in [1,d] \,.
$$
and $u = u(\eta) = (u_1(\eta),\ldots,u_j(\eta))$. Since $\{w_i\}$ is a basis and $s_j > 0$ for all $ j \in [1,d]$, the correspondence $u \leftrightarrow \eta$ is a bijection. By slight abuse of notation, we will thereby allow ourselves to interchange freely between $u$ and $\eta$. \\
\begin{claim} \label{claimcondu}
The vector $\eta$ satisfies the condition (\ref{condeta}) if the following condition is satisfied by the vector $u$:
\begin{equation} \label{aleph2}
\sum_{k: s_k = s_j} u_k^2 \leq \frac{R^2}{2} \sum_{ k: s_k = s_{j+1}  } u_k^2, ~~ \forall j \in \mathcal{J}.
\end{equation}
\end{claim}
\begin{proof}
Fix $j \in \mathcal{J}$. One has by definition
$$
s_{j+1} = \frac{1}{R} s_j
$$
and therefore
\begin{align*}
|| P_{E_j} \eta ||^2 ~& = \sum_{ \{k; s_k = s_j\} } \langle \eta, w_k \rangle^2
= \sum_{k: s_k = s_j} \frac{s_j^2}{R^2 s_{j+1}^2} \langle \eta, w_k \rangle^2 \\
& = \frac{1}{R^2 s_{j+1}^2} \sum_{k: s_k = s_j} u_k^2
 \stackrel{\eqref{aleph2}}{\leq} \frac{1}{2} \sum_{ k: s_k = s_{j+1}  } \frac{1}{s_{j+1}^2 } u_k^2 \\
& = \frac{1}{2} \sum_{ k: s_k = s_{j+1}  } \langle \eta, w_k \rangle^2 = \frac{1}{2} ||P_{E_{j+1}} \eta ||^2.
\end{align*}
\end{proof}
Define a function
$$
f_u(\eps) = \max_{ i \in I } \frac{\lambda_1(A_{\eta(u)} (\eps)^T M_i A_{\eta(u)} (\eps) )}{\mathrm{Tr}(A_{\eta(u)}(\eps)^T M_i A_{\eta(u)}(\eps))}.
$$
$$
~
$$
We claim that at this point, in order to prove the theorem, it is enough to prove the existence of a vector $u \in \RR^d$ satisfying:\\
($\aleph$) $\sum_{k: s_k = s_j} u_k^2 \leq \frac{R^2}{2} \sum_{k: s_k = s_{j+1}} u_k^2, ~~ \forall j \in \mathcal{J}$ and \\
($\beth$) One has $f_u(\eps) < f_u(0)$ in some non-degenerate interval $\eps \in (0, \eps_0]$. \\

To understand why this would indeed finish the proof, we recall that $A$ is defined as the minimizer of $f$ over $\DD_R$. Therefore, in order to get a contradiction, it is enough to construct a matrix $A' \in \DD_R$ such that $f(A') < f(A)$. If condition ($\aleph$) holds, then by Lemma \ref{lemcondeta} and Claim \ref{claimcondu} we know that there exists $\eps_0 > 0$ such that  $P(\eps) := \frac{A_{\eta(u)}(\eps)}{s_1(A_{\eta(u)}(\eps))} \in \DD_R$ for all $\eps \in [0, \eps_0]$. It is thus enough to show that $f(P(\eps)) < f(A)$ for some $\eps \in [0, \eps_0]$. Note, however, that for every $\alpha > 0$ one has $f(\alpha P(\eps)) = f(P(\eps))$ and therefore it is enough to show that $f(A_{\eta(u)} (\eps)) < f(A)$ for some $\eps \in [0, \eps_0]$. By the continuity of $f_i$ with respect to the topology $\mathcal{T}$, we have that $f_i(A_{\eta(u)}(\eps))$ is continuous with respect to $\eps$ at $\eps=0$ and thus there exists $\eps_1 > 0$ such that for all $i \notin I$ and for all $\eps \in [0, \eps_1]$ one has $f_i(A_{\eta(u)}(\eps)) < f(A)$. We conclude that it is enough to establish that there exists $\eps \in [0, \min(\eps_0, \eps_1)]$ such that
$\max_{i \in I} f_i(A_{\eta(u)}(\eps)) < f(A) $. But this follows immediately from the condition $(\beth)$. \\ \\

Let us sketch the idea for proving the existence of a vector satisfying the two above conditions.
We will first show that there exists a constant $c_0>0$ which depends only on the matrices $M_1,\ldots,M_\ell$ (and therefore does not depend on $R$) such that the following holds: for any given $R>0$ and any given $A \in \DD_R$ that minimizes $f(A)$, there exists a unit vector $\tilde{u} \in \RR^d$ and a number $\eps_0>0$ such that
\begin{equation} \label{needtoshow}
|u - \tilde{u}| < c_0 \Rightarrow f_u(\eps) < f_u(0), ~ \forall  \eps\in(0,\eps_0) \, .
\end{equation}
The point $\tilde{u}$ will be chosen in a manner analogous to the construction described above formula  \eqref{basicconstr}.

The second step will rely on the following fact (which we will see more clearly later): for any positive constant $c_0 > 0$ there exists a constant $R>0$ such that for any unit vector $\tilde{u} \in \RR^d$, the Euclidean ball centered at $\tilde{u}$ with radius $c_0$ contains a point $u$ which satisfies the condition ($\aleph$), regardless of the partition of the coordinates into the eigenspaces of $A$. A combination of this fact with the implication \eqref{needtoshow} will establish the existence of a vector satisfying both conditions. \\

In order to prove the implication \eqref{needtoshow}, we will need some additional perturbative estimates, contained in the following lemma:

\begin{lemma} \label{lemest}
There exist constants $C,c>0$ depending only on $M_1,\ldots,M_{\ell}$ such that the following holds: Let $R>0$ and let $A \in \DD_R$ be a positive definite matrix which minimizes $f(A)$. For any $\eta \in \RR^d$ and $u = u(\eta)$, one has \\
(i) For all $i \in I$ and $\eps > 0$,
$$
\mathrm{Tr}( A_\eta(\eps)^T M_i A_\eta(\eps) ) \geq \mathrm{Tr}(A M_i A) + c \eps |u|^2.
$$
(ii) For all $i \in I$ and $ j \in [1, k-1]$, there exists a vector $v_{i,j}$, satisfying $|v_{i,j}| < C$, such that
$$
\lambda_1(A_\eta(\eps)^T M_i A_\eta(\eps))) = \lambda_1(A M_i A) \left  (1 + \eps \sum_{m=1}^{k-1} \langle v_{i,m}, u \rangle^2 \right ) + o(\eps).
$$
\end{lemma}
\begin{proof}
For part (i) of the lemma, we simply estimate
$$
\mathrm{Tr} \bigl ( A_\eta (\eps)^T M_i A_\eta(\eps)  \bigr ) = \mathrm{Tr} \bigl ( ( \mathrm{Id} + \eps \eta \otimes \eta ) A M_i A ( \mathrm{Id} + \eps \eta \otimes \eta ) \bigr )
$$
$$
\geq \mathrm{Tr} (A M_i A) + \eps \mathrm{Tr} \bigl ( A M_i A \eta \otimes \eta + \eta \otimes \eta A M_i A \bigr ).
$$
Since for all $v \perp \eta$ one has
$$
\langle (A M_i A \eta \otimes \eta + \eta \otimes \eta A M_i A)v, v \rangle = 0
$$
we get
$$
\mathrm{Tr} \bigl ( A_\eta (\eps)^T M_i A_\eta(\eps)  \bigr ) \geq \mathrm{Tr}(A M_i A) + \frac{\eps}{ |\eta|^2 } \bigl \langle (A M_i A \eta \otimes \eta + \eta \otimes \eta A M_i A) \eta, \eta \bigr \rangle
$$
$$
= \mathrm{Tr}(A M_i M) + 2 \eps \langle A \eta, M_i A \eta \rangle = \mathrm{Tr}(A M_i A) + 2 \eps |M_i^{1/2} A \eta |^2
$$
where, in the last equality, we used the non-negativity of the matrices $M_i$. By the assumption that the matrices $M_i$ are non-degenerate, there exists a constant $c > 0$ depending on these matrices such that
$$
\mathrm{Tr} \bigl ( A_\eta (\eps)^T M_i A_\eta(\eps)  \bigr ) > \mathrm{Tr} (A M_i A) + c \eps |A \eta|^2
$$
and finally by the definition of $u$, we have
$$
|A \eta|^2 = \sum_{j=1}^d s_j^2 \langle \eta, w_j \rangle^2 = |u|^2.
$$
This finishes the proof of (i). In order to prove part (ii) of the lemma, fix $i \in I$ and let $F_i$ be the eigenspace corresponding to $\lambda_1(A M_i A)$. Remark that the assumption \eqref{Aisbad} together with the non-degeneracy of $A$ and $M_i$ imply that $\dim F_i \leq k-1$. Let $q_{i,1},\ldots,q_{i,k_i}$ be an orthonormal basis for this subspace, $k_i \leq k-1$. An application of Lemma \ref{eigenperturb} gives
$$
\lambda_1 \bigl (A_\eta(\eps)^T M_i A_\eta(\eps)) \bigr) = \lambda_1(A M_i A) \left (1 + \eps \sum_{m=1}^{k_i} \langle \eta, q_{i,m} \rangle^2 \right ) + o(\eps).
$$
By the definition of $u$, we have
$$
\langle \eta, q_{i,m} \rangle = \sum_{j=1}^d \frac{ \langle q_{i,m}, w_j \rangle}{s_j} s_j \langle \eta, w_j \rangle = \sum_{j=1}^d \frac{ \langle q_{i,m}, w_j \rangle}{s_j} u_j = \langle v_{i,m}, u \rangle
$$
where
$$
v_{i,m} := \left (\frac{ \langle q_{i,m}, w_1 \rangle}{s_1}, \ldots, \frac{ \langle q_{i,m}, w_d \rangle}{s_d} \right ).
$$
Combining the last three equations gives
$$
\lambda_1 \bigl (A_\eta(\eps)^T M_i A_\eta(\eps)) \bigr)
\lambda_1(A M_i A) \left  (1 + \eps \sum_{m=1}^{k_i} \langle v_{i,m}, u \rangle^2 \right ) + o(\eps).
$$
In order to prove the lemma, it remains to show that for all $i \in I$ and $m\in [1,k_i]$, the norm of the vector $v_{i,m}$  is bounded by a constant which only depends on $M_1,\ldots,M_d$.  To this end, fix $i \in I$, denote $\gamma = \lambda_1(A M_i A)$. We estimate for all $1 \leq m \leq k_i$,
$$
\frac{ |\langle q_{i,m}, w_j \rangle|}{s_j} = \gamma^{-1} \frac{|\langle A M_i A q_{i,m}, w_j \rangle|}{s_j} \leq
$$
$$
\gamma^{-1} ||A M_i||_{OP} \frac{ |q_{i,m}| |A w_j|}{s_j} = \gamma^{-1} \|A M_i\|_{OP}.
$$
Recall that by definition of $\DD_R$, we have  $\|A\|_{OP} = 1$. Moreover, according to the assumption \eqref{Aisbad},  since the matrix $A M_i A$ is positive definite, we have
$$
\gamma = \lambda_1(A M_i A) \geq \frac{1}{k} \mathrm{Tr}(A M_i A) \geq \frac{1}{k} \langle M_i A w_1, A w_1 \rangle
$$
$$
= \frac{1}{k} \langle M_i w_1, w_1 \rangle \geq  \frac{1}{k} \lambda_d(M_i).
$$
Consequently,
$$
\frac{ |\langle q_{i,m}, w_j \rangle|}{s_j} \leq k \frac{ \lambda_1( M_i ) }{ \lambda_d(M_i) }.
$$
The right hand side depends only on $M_i$, and therefore so does the upper bound on the norm of the vector $v_{i,m}$, as promised. The proof is complete.
\end{proof}
\bigskip
In order to use the above lemma, we will define
$$
Q(u) = \sum_{i \in I } \sum_{m=1}^{k-1} \langle u, v_{i,m} \rangle^2.
$$
Recall that one of the assumptions of the theorem is that $\ell (k-1) \leq d-1$ and that $|I| \leq \ell$. Consequently, the quadratic form $Q(u)$ is degenerate and thus there exists a unit vector $\tilde{u}$ (which is chosen to satisfy condition \eqref{needtoshow}) such that $Q(\tilde{u}) = 0$. \\ \\
Our goal at this point is to find a constant $c_0$ such that implication \eqref{needtoshow} holds. Fix $i \in I$ and define
$$
B_u(\eps) = A_{\eta(u)}(\eps)^T M_i A_{\eta(u)}(\eps).
$$
We have
\begin{equation} \label{diffratio}
\frac{\lambda_1 \bigl (B_u(\eps) \bigr ) }{\mathrm{Tr} \bigl (B_u(\eps)\bigr)} - \frac{\lambda_1(B_u(0) )}{\mathrm{Tr}(B_u(0))} =
\end{equation}
$$
\frac{\lambda_1(B_u(0))}{ \mathrm{Tr} \bigl (B_u(\eps) \bigr ) } \left( \frac{ \lambda_1 \bigl (B_u(\eps) \bigr ) - \lambda_1(B_u(0) )} {\lambda_1(B_u(0))} - \frac{ \mathrm{Tr} \bigl (B_u(\eps) \bigr) - \mathrm{Tr}(B_u(0)) }{ \mathrm{Tr}(B_u(0)) } \right  ).
$$
By making the assumption that $c_0 < 1/2$, we have
$$
|u - \tilde{u}| < c_0 \Rightarrow |u|^2 \geq \frac{1}{2}.
$$
An application of part (i) of Lemma \ref{lemest} now gives
$$
|u - \tilde{u}| < c_0 \Rightarrow \mathrm{Tr} \bigl (B_u(\eps) \bigr) - \mathrm{Tr}(B_u(0)) \geq c_1 \eps.
$$
Thanks to the fact that $||A||_{OP} \leq 1$, we have $\mathrm{Tr}(B_u(0)) < C_1$, where $C_1$ only depends on $M_i$. Consequently, we conclude that
$$
|u - \tilde{u}| < c_0 \Rightarrow \frac{ \mathrm{Tr}(B_u(\eps) \bigr) - \mathrm{Tr}(B_u(0)) }{ \mathrm{Tr}(B_u(0)) } > c_2 \eps.
$$
(again, $c_2$ is a constant depending only on $M_i$). Inspecting equation \eqref{diffratio}, we note that the term outside the brackets on the right hand side is positive (since the matrix $B(\eps)$ is positive definite). Consequently, in order to establish \eqref{needtoshow}, it is enough to prove that there exists a constant $ c_0 \in (0,1/2)$ and $\eps_1 > 0$ depending only on $M_i$ such that
$$
|u - \tilde{u}| < c_0 \Rightarrow \frac{ \lambda_1 \bigl (B_u(\eps) \bigr ) - \lambda_1(B_u(0) )} {\lambda_1(B_u(0))} \leq c_2 \eps, ~~
\forall  \eps \in (0,\eps_1) \, .
$$
An application of part (ii) of Lemma \ref{lemest} yields
$$
\frac{ \lambda_1 \bigl (B_u(\eps) \bigr ) - \lambda_1(B_u(0) )} {\lambda_1(B_u(0))} = \eps \sum_{m=1}^{k-1} \langle u, v_{i,m} \rangle^2 + o(\eps).
$$
The two above equations combined imply that it is enough to show that
$$
|u - \tilde{u}| < c_0 \Rightarrow \sum_{m=1}^{k-1} \langle u, v_{i,m} \rangle^2 < c_2
$$
and since, according to the same lemma, we have $|v_{i,m}| < C$, it follows that it is enough to require that
$$
|u - \tilde{u}|^2 < C^{-2} c_2 (k-1)^{-1}.
$$
By repeating this argument for all $i=1,\ldots,\ell$ and taking the intersection of the corresponding neighborhoods of $\tilde{u}$, we conclude that \eqref{needtoshow} holds. We may conclude the above in the following:
\begin{proposition} \label{propnts}
There exists a constant $0<c_0<1$ that depends only on the matrices $M_1,\ldots,M_\ell$ such that the following holds: for any given $R>0$ and any given $A \in \DD_R$ that minimizes $f(A)$, there exists a unit vector $\tilde{u} \in \RR^d$ and a number $\eps_0>0$ such that for every $u \in \RR^d$ satisfying $|u - \tilde{u}| < c_0$, one has
\begin{equation}
f_u(\eps) < f_u(0), ~ \forall   \eps \in (0, \eps_0) \, .
\end{equation}
\end{proposition}

In order to finish the proof, we will need the following simple lemma:
\begin{lemma} \label{lemR}
For any positive constant $0 < c_0 < 1$ there exists a constant $R>0$ such that for any unit vector $\tilde{u} \in \RR^d$, the Euclidean ball centered at $\tilde{u}$ with radius $c_0$ contains a point $u$ that satisfies
$$
\sum_{1 \leq k \leq d} u_k^2 \leq \frac{R^2}{2} \min_{1 \leq k \leq d} u_k^2.
$$
\end{lemma}
\begin{proof}
For $v \in \RR^d$ we denote by $B(v)$ the closed ball of radius $c_0$ around $v$. We define two functions $g,h:\RR^d \to \RR$ by
$$
g(u) = \min_{1 \leq k \leq d} u_k^2
$$
and
$$
h(v) = \max_{u \in B(v)} g(u).
$$
Define $\delta = \min_{|v|=1} h(v)$. It is easy to verify that $h(v)$ is a continuous function and thus the minimum is attained, and moreover that $\delta > 0$. Finally, we can choose $u = \arg \max_{u \in B(\tilde u)} g(u)$. We get
$$
\sum_{1 \leq k \leq d} u_k^2 \leq 4d \leq \frac{4d}{\delta} \min_{1 \leq k \leq d} u_k^2.
$$
\end{proof}

We are finally ready to prove our theorem.

\begin{proof}[Proof of Theorem 1]
Let $c_0=c_0(M_1,\ldots,M_\ell)$ be the constant from Proposition \ref{propnts}. Let $R=R(c_0)$ be the constant corresponding to $c_0$ obtained by an application of Lemma \ref{lemR}. Let $A$ be the positive definite minimizer of $f$ in $\DD_R$. Proposition \ref{propnts} ensures the existence of a unit vector $\tilde{u}$ such that any $u$ that satisfies $|u-\tilde{u}| < c_0$, must also satisfy condition ($\beth$). Now, by Lemma \ref{lemR}, there exists a point $u$ satisfying $|u - \tilde{u}| < c_0$ which satisfies condition ($\aleph$). We have therefore found a point $u \in \RR^d$ which satisfies both conditions, and the proof is complete.
\end{proof}

\bigskip
\begin{proof}[Proof of lemma \ref{eigenperturb}]
Since the matrix $B$ is positive definite, we can write $B = U^T D U$ where $U$ is an orthogonal matrix and $D$ is diagonal. Moreover, we may clearly assume that the sequence of diagonal entries of $D$ is the sequence $\lambda=\{\lambda_1,\ldots,\lambda_d \}$ (hence, it is non-increasing). Next, note that for $v \in \RR^d$,
$$
(\eta \otimes \eta) U^T v = \eta \langle U \eta, v \rangle = U^T \Bigl((U \eta) \otimes (U \eta)\Bigr) v
$$
so
$$
(\Id + \eps \eta \otimes \eta) U^T = U^T (\Id + \eps (U \eta) \otimes (U \eta))
$$
and a similar calculation yields
$$
U (\Id + \eps \eta \otimes \eta) = (\Id + \eps (U \eta) \otimes (U \eta)) U.
$$
We get,
$$
\lambda_j \bigl( B(\eps) \bigr) = \lambda_j \bigl( (\mathrm{Id} + \eps \eta \otimes \eta)^T U^T D U (\mathrm{Id} + \eps \eta \otimes \eta) \bigr) =
$$
$$
\lambda_j \bigl( U^T (\mathrm{Id} + \eps \xi \otimes \xi)^T D (\mathrm{Id} + \eps \xi \otimes \xi) U \bigr) =
$$
$$
\lambda_j \bigl( (\mathrm{Id} + \eps \xi \otimes \xi)^T D (\mathrm{Id} + \eps \xi \otimes \xi) \bigr)
$$
where $\xi = U \eta$.
Define
$$
F(\lambda, \eps) = \det \left ((\mathrm{Id} + \eps \xi \otimes \xi)^T D (\mathrm{Id} + \eps \xi \otimes \xi) \right ).
$$
Let $\{e_i\}_{i=1}^d$ be the standard basis of $\RR^d$. We claim that
\begin{equation} \label{detpert}
F(\lambda, \eps) = \prod_{i=1}^d \left (\lambda_i + 2 \eps \lambda_i \langle \xi, e_i \rangle^2 - \lambda \right) + O(\eps^2).
\end{equation}
Indeed, we can write
$$
F(\lambda, \eps) = \det \Bigl (D + \eps (D (\xi \otimes \xi) + (\xi \otimes \xi) D) + \eps^2 \xi \otimes \xi D \xi \otimes \xi - \lambda \mathrm{Id} \Bigr).
$$
Now, when expressing the left hand side determinant as a sum of products of entries, observe that each summand except for the principal diagonal is of the order $O(\eps^2)$. Finally, observe that
$$
(D (\xi \otimes \xi))_{j,j} = ((\xi \otimes \xi) D)_{j,j} = \lambda_j \langle \xi, e_j \rangle^2
$$
and formula \eqref{detpert} follows.
We will first assume that $\lambda_i$ has multiplicity $1$. Under this assumption,  by differentiating formula \eqref{detpert}, we have
$$
\left. \frac{\partial}{\partial \lambda} F(\lambda, \eps)  \right |_{\lambda = \lambda_i, \eps=0} =
- \prod_{1 \leq j \leq d \atop j \neq i} (\lambda_j - \lambda_i) \,,
$$
and
$$
\left. \frac{\partial}{\partial \eps} F(\lambda, \eps)  \right |_{\lambda = \lambda_i, \eps=0} =
2 \lambda_i \langle \xi, e_i \rangle^2 \prod_{1 \leq j \leq d \atop j \neq i} (\lambda_j - \lambda_i) \,.
$$
The inverse function theorem then gives
$$
\lambda_i ( B(\eps) ) = \lambda_i + 2 \eps \lambda_i \langle \xi, e_i \rangle^2 + O(\eps^2).
$$
Noting that $\langle \xi, e_i \rangle = \langle U^T e_i, \eta \rangle$ and that by definition of the matrix $U$, the vector $U^T e_i$ is an eigenvector of $B$ corresponding to the eigenvalue $\lambda_i$, we conclude that $\langle \xi, e_1 \rangle^2 = |P_{E_i} \eta|^2$. This finally gives
\begin{equation} \label{eqten}
\lambda_i ( B(\eps) ) = \lambda_i + 2 \eps \lambda_i |P_{E_i} \eta|^2 + O(\eps^2),
\end{equation}
which is equation \eqref{eqeigenderiv}.

If $E_i \subset \eta^\perp$, then $E_i$ remains an eigenspace of $B(\eps)$ corresponding to the eigenvalue $\lambda_i$, for all $\eps>0$.
In this case, the lemma clearly holds.

It remains to verify the case in which the multiplicity of $\lambda_i$ is at least $2$ and $E_i$ is not a subset of $\eta^\perp$. In this case, we define
$$
E_j' = E_j \cap \eta^\perp,
$$
$$
J = \{ j; ~~ \lambda_j \mbox{ has multiplicity greater than one } \} \,,
$$
and
$$
E = \bigoplus_{j \in J} E_j'.
$$
It is straightforward to check that the subspaces $E$ and $E^\perp$ are invariant under both $B$ and $\eta \otimes \eta$, and consequently
$$
B(\eps) = B'(\eps) \oplus B''(\eps)
$$
where
$$
B'(\eps) = P_{E^\perp} B(\eps) P_{E^\perp}, ~~ B''(\eps) = P_{E} B(\eps) P_{E}.
$$
Now, by definition of the subspace $E$, we see that $B'(\eps)$ does not depend of $\eps$ and $B''(0)$ has distinct eigenvalues. Moreover, since we assume that $\lambda_i$ has multiplicity at least two, it follows that $E_i \cap \eta^\perp$ has co-dimension 1 in $E_i$. Therefore, the multiplicity of $\lambda_i$ as an eigenvalue of $B''(0)$ is exactly $1$. Let $j$ be such that $\lambda_j(B''(0)) = \lambda_i$. Then we have, by \eqref{eqten},
$$
\lambda_j ( B(\eps)'' ) = \lambda_i + 2 \eps \lambda_i |P_{F} \eta|^2 + O(\eps^2)
$$
where $F$ is the $B''(0)$-eigenspace corresponding to $\lambda_j$. Consequently, we have
$$
\lambda_{i_0} ( B(\eps) ) = \lambda_i + 2 \eps \lambda_i |P_{F} \eta|^2 + O(\eps^2).
$$
where $i_0 = \min \{ K : \lambda_k(B(0)) = \lambda_i\}$. Finally, by the definition of $E$, we have $F = E^\perp \cap E_i = \mathrm{span} \{P_{E_i} \eta\}$ which gives $P_{F} \eta = P_{E_i} \eta$. To conclude, we have that if $\lambda_i < \lambda_{i-1}$, then
$$
\lambda_{i} ( B(\eps) ) = \lambda_i + 2 \eps \lambda_i |P_{E_i} \eta|^2 + O(\eps^2) \,
$$
and otherwise
$$
\lambda_{i} ( B(\eps) ) = \lambda_i \,.
$$
The lemma is complete.
\end{proof}

\subsection{Sharpness of the bound} \label{sec:sharp}

In order to show that Theorem \ref{mainthm} is sharp, we would like to prove \eqref{eqmainsharp}. For the sake of simplicity, we will only demonstrate this in the case that $(k-1) \ell = d$.

Fix $\eps < \tfrac{1}{d}$. For $1 \leq i \leq \ell$, let $M_i$ be the matrix
%$$
%M_i = \mathrm{diag} \left  ({ \underbrace{\eps,\ldots,\eps} \atop {\tiny \mbox {(k-1)(i-1) times }}}, { \underbrace{1,\ldots,1} \atop {\tiny \mbox {(k-1) times %}}}, { \underbrace{\eps,\ldots,\eps} \atop {\tiny \mbox {d-(k-1)i times }}}  \right  ).
%$$
$$
M_i = \mathrm{diag} \left({{  \mbox { \tiny ($k$-1)($i$-1) times }} \atop  \overbrace{\eps,\ldots,\eps} \, ,  }
 { {  \mbox{\tiny ($k$-1) times }} \atop \overbrace{1,\ldots,1} \, , }  { {  \mbox {\tiny $d$-($k$-1)$i$ times }} \atop \overbrace{\eps,\ldots,\eps} }  \right  ).
$$

Let $A$ be an arbitrary $d \times d$ matrix. Denote by $v_1,\ldots,v_n$ the rows of $A$. Since the left hand side of equation \eqref{eqmainsharp} is invariant under multiplication of $A$ by a scalar, we may assume that $\max_{1 \leq j \leq n} |v_j|^2 = 1$.

Fix $J$ so that $|v_J|=\max_j |v_j|$ and for every $1 \leq i \leq \ell$, set
$$
I(i) = \{ (k-1)(i - 1) + 1, \ldots, (k-1) i \}.
$$
We have
\begin{align*}
\mathrm{Tr}( A^T M_i A ) = &~ \sum_{j=1}^d \langle e_j, A^T M_i A e_j \rangle \\
= &~ \sum_{j=1}^d |M_i^{1/2} A e_j|^2 \\
= &~ \sum_{j_1=1}^d \sum_{j_2=1}^d \langle M_i^{1/2} A e_{j_1}, e_{j_2} \rangle^2  \\
= &~\sum_{j_1=1}^d \sum_{j_2=1}^d \langle A e_{j_1}, M_i^{1/2} e_{j_2} \rangle^2  \\
= &~ \sum_{j_1=1}^d \left ( \sum_{j_2 \in I(i)} \langle e_{j_1}, A^T e_{j_2} \rangle^2 + \eps \sum_{j_2 \in [d] \setminus I(i)} \langle e_{j_1}, A^T e_{j_2} \rangle^2 \right ) \\
= &~ \sum_{j \in I(i)} |v_{j}|^2 + \eps \sum_{j \in [d] \setminus I(i)} |v_{j}|^2 \\
\leq &~ k-1 + d \eps < k .
\end{align*}
Now, let $i_0$ be an integer such that $J \in I(i_0)$. Since $|v_J| = \max_j |v_j| = 1$, we have
$$
\lambda_1( A^T M_{i_0} A) \geq \langle A^T M_{i_0} A v_J^T, v_J^T \rangle = \langle M_{i_0} A v_J^T, A v_J^T \rangle \stackrel{J \in  I(i_0)}\geq \langle A v_J^T, e_J \rangle^2 = |v_J|^4 = 1.
$$
The last two inequalities give $\frac{\lambda_1(A^T M_{i_0} A) }{ \mathrm{Tr}( A^T M_{i_0} A ) } > \frac{1}{k}.$

\section{Self-Interacting Random walks}

The goal of this section is to establish Theorem \ref{cor1}. Our first ingredient for its proof will be the following lemma, which generalizes \cite[Lemma 2.2]{PPS}. However stated in a slightly more general form, the proof of this lemma follows similar steps to the proof which appears in \cite{PPS}. We omit it here.

\begin{lemma} \label{lemsubmartingale}
Let $k \geq 2, \beta > 0$. Let $Z$ be a mean-zero random vector in $\RR^d$ with $2 + \beta$ moments, whose respective covariance matrix $M$ satisfies
\begin{equation} \label{tracealpha}
\mathrm{Tr}(M) > k \lambda_1(M).
\end{equation}
Then there exist constants $r_0>0$ and $\alpha > k-2$ such that if $\vert x \vert \geq r_0$ then
$$
\EE[ \varphi(x + Z) - \varphi(x)] \leq 0
$$
where $\varphi(x) = \min(\vert x \vert^{-\alpha}, 1)$.
\end{lemma}

\begin{proof} (sketch).
Following the same lines as \cite[Lemma 2.2]{PPS}, a Taylor expansion of the function $\varphi(\cdot)$ around the point $x$ gives that
$$
\EE[ \varphi(x + Z) - \varphi(x)] \leq \sum_{i=1}^d \frac{\alpha x_i^2 (\lambda_i (\alpha + 2) - \sum_{j=1}^d \lambda_j )}{\vert x \vert^{\alpha + 4}} + O \left (  \vert x \vert^{-\min(\alpha + 3, \alpha + \beta + 2 )} \right ).
$$
The assumption \eqref{tracealpha} gives that if $\alpha-k+2$ is small enough, then for some $c>0$ we have
$$
\sum_{i=1}^d \frac{\alpha x_i^2 (\lambda_i (\alpha + 2) - \sum_{j=1}^d \lambda_j )}{\vert x \vert^{\alpha + 4}} \leq -c \alpha \vert x \vert^{-(\alpha + 2)}.
$$
Combining the last two formulas finishes the proof.
\end{proof}

The second ingredient we need is well-known to experts. We provide the proof for completeness.
\begin{lemma} \label{lemball}
Let $\mu_1,\ldots,\mu_\ell$ be centered probability measures in $\RR^d$, with non-singular and finite covariance matrices. Then there exists a constant $C>0$ such that if $\{X_t\}$ is an adaptive random walk using $\mu_1,\ldots,\mu_\ell$ , then for every $R>C$, for all $x \in \RR^d$ and for all $\delta > 0$, we have
\begin{equation} \label{eq:nts1}
\PP \bigl(x+X_i \in B(0,R), ~~ \forall   i \in [1,  R^{2 + \delta}] \bigr) \leq C e^{-\frac{R^{\delta}}{C}}.
\end{equation}
\end{lemma}
\begin{proof}
Define $c = \min_{1 \leq i \leq \ell} \mathrm{Tr}(M_i)$ and $K = \max_{1 \leq i \leq \ell} \mathrm{Tr}(M_i)$. For all $i \in [1,\ell]$, let $Z_i$ be a random variable with law $\mu_i$, so that $c \le \EE[ \vert Z_i \vert^2 ] \le K$ for all $i \in [1,\ell]$.
Fix a constant $Q > 0$ whose value will be determined later on. Observe that, by Cauchy-Schwartz,
\begin{equation}  
\EE \left [ \vert Z_i \vert  \mathbf{1}_{\{|Z_i| > QR \} } \right ] \le \frac{K}{QR}, ~~ \forall i \in [1, \ell] \,.
\end{equation}
Since $\EE(Z_i)=0$, it follows that 
\begin{equation} \label{eq:TrRad}
\Bigl|\EE \left [  Z_i   \mathbf{1}_{\{|Z_i| \le QR \} } \right ] \Bigr| \le \frac{K}{QR}, ~~ \forall i \in [1, \ell] \,.
\end{equation}
By dominated convergence,  there exists a constant $R_0>0$ such that
for $R>R_0$ we have
\begin{equation} \label{eq:TrRad2}
\EE \left [ \vert Z_i \vert^2 \mathbf{1}_{\{|Z_i| \le Q R \} } \right ] > c/2, ~~ \forall i \in [1, \ell] \,.
\end{equation}

Given $R$,  define $Y_t = X_{t+1}-X_t$ and
$$
\tilde Y_{t} = \mathbf{1}_{\{ |Y_t| \leq QR \} } Y_t
$$
for all $t$. The conditional means
$$
w_t = \left . \EE \left [\tilde Y_t \right | \mathcal{F}_t \right ] \,
$$
 satisfy $|w_t| \le \frac{K}{QR}$ by (\ref{eq:TrRad}) and $\tilde Y_{t}$ satisfy
\begin{equation} \label{eq:incr}
\EE \left . \left [\vert \tilde Y_t \vert^2 \right | \mathcal{F}_t \right ]  \ge c/2
\end{equation}
by (\ref{eq:TrRad2}). Next, consider the partial sums $\tilde X_t = X_0 + \sum_{j=0}^{t-1} \tilde Y_j$. On the event $|\tilde X_t| \le R$, we have
\begin{equation}  \label{eq:incr2}
\EE \left . \left [\vert \tilde X_{t+1} \vert^2 -\vert \tilde X_t \vert^2 \right | \mathcal{F}_t \right ]  =
\EE \left . \left [\vert \tilde Y_t \vert^2 \right | \mathcal{F}_t \right ] + 2\langle w_t, \tilde X_t \rangle \ge c/2-2|w_t|R \ge c/2-\frac{2K}{Q}=c/4
\end{equation}
if we pick $Q=8K/c$.

Fix a point $x \in \RR^d$ and consider the stopping time
$$
\tau = \inf \{t; ~x+ \tilde X_t \notin B(0,R) \}.
$$
By (\ref{eq:incr2}), the process $S_t = \vert \tilde X_{t\wedge \tau} \vert^2 - c ({t\wedge \tau})/4$ is a submartingale.
The optional stopping theorem gives
$$
0 \leq \EE [S_{\tau \wedge t} ] = \EE \left [\vert \tilde X_{\tau \wedge t} \vert^2 \right ] - c ( \EE[ \tau \wedge t]) / 4 \leq 
 (1+Q)^2 R^2 - c\EE (\tau \wedge t) / 4 \,,
$$
where  we used the inequality $\vert \tilde Y_s \vert \leq QR$, and that for all $s < \tau$, we have
$\vert \tilde X_s \vert \leq R$. By taking $t \to \infty$, we get
$$
\EE[\tau] \leq 4(1+Q)^2 c^{-1} R^2=Q'R^2,
$$
where $Q'= 4(1+Q)^2 c^{-1}$. By Markov's inequality
$$
\PP(\tau > 2Q' R^2) < 1/2 \,.
$$

Note that if $X_{t} \in B(0,R)$ and  $X_{t+1} \in B(0,R)$, then  $\vert Y_t \vert \le 2R \le QR$. Thus we conclude that
$$
\PP \left ( x+X_t \in B(0,R), \forall   t \in [1, 2Q' R^2] \right ) \leq 1/2.
$$
Since $x$ was arbitrary,  induction yields that for every integer $k \geq 1$,
$$
\PP \bigl(X_i \in B(0,R), ~~ \forall i \in [1,   2k Q' R^2] \bigr) \leq 2^{-k} \,.
$$
The proof is complete.
\end{proof}

\begin{proof} [Proof of Theorem \ref{cor1}]
By an application of Theorem \ref{mainthm}, there exists a matrix $A$ and a $\delta > 0$ such that
$$
\frac{\lambda_1(\tilde M_i) }{ \mathrm{Tr}( \tilde M_i ) } \leq \frac{1}{k+\delta}
$$
for all $1 \leq i \leq \ell$, where $\tilde M_i = A^T M_i A$. Define $Y_t = A X_t$ and observe that $\{Y_t\}$ is an adaptive random walk using the measures $\tilde \mu_1,\ldots, \tilde \mu_\ell$, defined as the push-forward under the matrix $A$ of the measures $\mu_1,\ldots,\mu_\ell$. Moreover, note that the covariance matrix of $\tilde \mu_i$ is $\tilde M_i$ for all $i$.

Let $\tilde R$ be the diameter of the ellipsoid $A B(0,R)$. By definition, we have that
\begin{equation} \label{BallXY}
Y_t \notin B(0,\tilde R) \Rightarrow X_t \notin B(0,R).
\end{equation}
%It therefore suffices to consider the probability of $Y_t$ to stay outside of $B(0, \tilde R)$.

Given $\eps \in (0,0.1)$, let $\alpha = (k - 2) + k\eps$ and $N_t = \vert Y_t \vert^{-\alpha}$.
According to Lemma \ref{lemsubmartingale}, if $\eps$ is sufficiently small, then there exists $R_1 > 0$    such that
\begin{equation} \label{supermartingale}
\vert Y_t \vert \geq R_1 \Rightarrow \EE[N_{t+1} | \mathcal{F}_t] \leq N_t \,.
\end{equation}
Next, define $R_2 = \max(\tilde R, R_1)$ and
 $r = \max \left (R_2, T^{(1 - \eps/2)/2} \right )$. Let
$$
\tau_1 = \min \{t; ~ \vert Y_t \vert \geq r \}.
$$
By Lemma \ref{lemball}, there exist constants $c,C_1>0$ which do not depend on $T$, such that
$$
\PP( \tau_1 >  r^{2 + \eps}) \leq C_1 e^{-c r^{\eps}}.
$$
By the definition of $r$  and the inequality $(1-\eps /2)(2+\eps) \leq 2$, we infer that
\begin{equation} \label{eqtau1}
\PP( \tau_1 >  \max(T, R_2^{2+\eps}) ) \leq C_1 e^{-c T^{\eps/4}}.
\end{equation}
Next, consider the stopping time
$$
\tau_2 = \inf \{t \geq \tau_1; ~ \vert Y_t \vert \leq R_2  \}
$$
with the conventions that $\inf \emptyset = \infty$ and $N_\infty=0$. By \eqref{supermartingale}, we see that   $\{N_{(t+\tau_1) \wedge \tau_2}\}$ is a supermartingale. Thus, by the optional stopping theorem and Fatou's lemma, we have that
$$
\PP(\tau_2 < \infty) \leq \EE[R_2^{\alpha} N_{\tau_2}]  \le \EE[R_2^{\alpha} N_{\tau_1}] \leq (R_2/r)^{\alpha} \leq C T^{-(k-2+k\eps)  (1-\eps/2)/2} \leq C T^{\frac{-(k-2)}{2} - \delta}
$$
for some $C,\delta > 0$ which do not depend on $T$ (in the last inequality we use the assumption that $\eps < 0.1$). Combining this with \eqref{eqtau1} and using a union bound, gives
$$
\PP \left (\exists t > \max(T, R_2^{2+\eps}) \mbox{ such that } Y_t \in B(0, \tilde R) \right ) \leq \PP \left ( \tau_1 >  \max(T, R_2^{2+\eps}) \right ) + \PP(\tau_2 < \infty)
$$
$$
 \leq C_1 e^{-c T^{\eps/4}} + C T^{\frac{-(k-2)}{2} - \delta}.
$$
Since $C_1, C,\eps,\delta$ and $R_2$ do not depend on $T$, the above implies that there exists a constant $C_2>0$ which does not depend on $T$ such that
$$
\PP \left (\exists t>T \mbox{ such that } Y_t \in B(0, \tilde R) \right ) < C_2  T^{\frac{-(k-2)}{2} - \delta}.
$$
In view of equation \eqref{BallXY}, this gives
$$
\PP \left (\exists t>T \mbox{ such that } X_t \in B(0, R) \right ) < C_2  T^{\frac{-(k-2)}{2} - \delta}
$$
and the proof is complete.
\end{proof}

\section*{Acknowledgements}

This work was done when R.E. was a visiting researcher at the Theory Group of Microsoft Research. He wishes to thank the Theory Group for their kind hospitality.


\begin{thebibliography}{GGM}
\bibitem[K76]{Kato}
T. Kato. {\it Perturbation theory for linear operators.} Second Edition. Springer, Berlin (1976).

\bibitem[HJ12]{HJ}
R.A. Horn and C.R. Johnson. {\it Matrix analysis.} Second Edition. Cambridge University Press (2013).

\bibitem[PPS13]{PPS}
Y. Peres, S. Popov and P. Sousi. {\it On recurrence and transience of self-interacting random walks.} Bulletin of the Brazilian Mathematical Society, (New Series) 44 p.\ 841--867 (2013).

\end{thebibliography}
\end{document}